\documentclass[11pt]{article}

\usepackage[a4paper,margin=25mm]{geometry}
\usepackage[T1]{fontenc}
\usepackage[utf8]{inputenc}
\usepackage{lmodern}
\usepackage{amsmath,amssymb,amsthm,mathtools}
\usepackage{booktabs,array,enumitem,microtype}
\usepackage{xcolor}
\usepackage{authblk}
\usepackage{float}

\definecolor{numbercolor}{RGB}{0,80,150}

\usepackage[
colorlinks=true,
linkcolor=numbercolor,
citecolor=green!40!black,
urlcolor=violet!70!black
]{hyperref}

\usepackage[noabbrev]{cleveref}

\makeatletter

\renewcommand{\@author}{%
	\def\rlap##1{##1}%
	\AB@authlist
}

\renewcommand{\AB@affilnote}[1]{%
	\textsuperscript{\normalfont#1}\,%
}

\newcommand{\makefirstpageinfo}[1]{%
	\begingroup
	\renewcommand{\thefootnote}{}%
	\renewcommand{\@makefntext}[1]{\noindent##1}%
	\footnotetext[0]{%
		\begingroup
		\let\\\par
		\AB@affillist\par
		\endgroup
		\smallskip
		\textit{2020 Mathematics Subject Classification.} #1%
	}%
	\endgroup
}

\makeatother

\newtheorem{theorem}{Theorem}[section]

\newtheorem{lemma}[theorem]{Lemma}
\newtheorem{corollary}[theorem]{Corollary}

\newtheorem{remark}[theorem]{Remark}

\crefname{theorem}{theorem}{theorems}
\Crefname{theorem}{Theorem}{Theorems}
\crefname{proposition}{proposition}{propositions}
\Crefname{proposition}{Proposition}{Propositions}
\crefname{lemma}{lemma}{lemmas}
\Crefname{lemma}{Lemma}{Lemmas}
\crefname{corollary}{corollary}{corollaries}
\Crefname{corollary}{Corollary}{Corollaries}
\crefname{definition}{definition}{definitions}
\Crefname{definition}{Definition}{Definitions}
\crefname{remark}{remark}{remarks}
\Crefname{remark}{Remark}{Remarks}
\newtheorem{question}[theorem]{Question}

\newcommand{\cc}{\mathbf{c}}
\newcommand{\e}{\mathbf{e}}
\newcommand{\zz}{\mathbf{z}}
\newcommand{\y}{\mathbf{y}}
\newcommand{\w}{\mathbf{w}}
\newcommand{\x}{\mathbf{x}}
\newcommand{\q}{\mathbf{q}}
\newcommand{\p}{\mathbf{p}}
\newcommand{\vv}{\mathbf{v}}
\newcommand{\uu}{\mathbf{u}}
\newcommand{\dd}{\mathbf{d}}
\newcommand{\F}{\mathbb F_3}
\newcommand{\wt}{\operatorname{wt}}

\newcommand{\rank}{\operatorname{rank}}
\newcommand{\Min}{\operatorname{Min}}

\newcommand{\Zcyc}[1]{\mathbb Z/#1\mathbb Z}
\title{Irreducible rootless unimodular lattices
	generated\\ by vectors of norm $3$}

\author[a]{Hong-Jun Ge\thanks{%
		Corresponding author.
		Email:
		\href{mailto:gehj22@mail.ustc.edu.cn}
		{\nolinkurl{gehj22@mail.ustc.edu.cn}}.}}
\author[a,b]{Jack H. Koolen}
\author[a]{Jing-Yuan Liu}
\author[c]{Kiyoto Yoshino}

\affil[a]{%
	School of Mathematical Sciences,
	University of Science and Technology of China,
	Hefei, 230026, People's Republic of China
}

\affil[b]{%
	CAS Wu Wen-Tsun Key Laboratory of Mathematics,
	University of Science and Technology of China,
	Hefei, 230026, People's Republic of China
}

\affil[c]{%
	Department of Information Science, Faculty of Science, 
	Toho University, 
	2-2-1 Miyama, Funabashi, Chiba 274-8510, Japan
}

\date{\today}

\begin{document}

\maketitle

\makefirstpageinfo{%
	Primary 11H06;
	Secondary 11H71, 94B05, 05C22, 05C50.%
}

\begin{abstract}
	The classical ADE classification implies that $E_8$
	is the unique irreducible unimodular Euclidean lattice generated
	by vectors of norm $2$.
	In sharp contrast, for every positive integer $m$, we construct an
	irreducible odd unimodular lattice of rank $24m$
	and minimum norm $3$ that is generated by its minimal vectors.
\end{abstract}

\section{Introduction}
Throughout, all lattices are Euclidean, and the norm of a lattice
vector is its squared Euclidean length.
Generation of lattices is always understood over $\mathbb Z$.

The classical ADE classification of root lattices states that every irreducible integral lattice generated by vectors of
norm $2$ is isometric to one of
$A_n\ (n\ge1)$, $D_n\ (n\ge4)$, $E_6$, $E_7$, and $E_8$;
see \cite{Cartan1894,Witt}.
Consequently, up to isometry, $E_8$ is the unique
irreducible unimodular lattice generated by vectors of norm $2$.

In this note, we show that the situation changes sharply at minimum norm \(3\).
Even under the same requirements of irreducibility, unimodularity, and generation by minimal vectors, the rank is no longer bounded.
More precisely, our main
result is the following.

\begin{theorem}\label{thm:main lattice}
	For every positive integer $m$, there exists an irreducible odd
	unimodular lattice $L_m$ of rank $24m$ and minimum norm $3$
	that is $3$-integrable and generated by its minimal vectors.
\end{theorem}

Classical examples include the shorter Leech, odd Leech, and
Conway--Borcherds lattices of ranks $23$, $24$, and $26$, respectively
\cite{Borcherds1984,ConwaySloane}.
The irreducible unimodular lattices of minimum norm $3$ in ranks
$27$ and $28$ were classified by Bacher and
Venkov~\cite{BacherVenkov2001}, and those of rank $29$ by
Allombert and Chenevier~\cite{AllombertChenevier2025}.
All these lattices are generated by their minimal vectors;
see \cite{MartinetBatut} and
\cite[Example~4.3]{AllombertChenevier2025}.
This property of generation is not automatic for unimodular lattices of
minimum norm $3$. For example,
Chenevier~\cite[Proposition~8.2]{Chenevier2025}
gives a family of such lattices whose minimal vectors generate a proper
finite-index sublattice in sufficiently large ranks.

Our construction cyclically glues copies of an index-$3$ sublattice
of the odd Leech lattice.
The constituent sublattices are generated by vectors of norm $3$,
and the added gluing vectors have the same norm.
The key point is that this gluing introduces no vectors of smaller norm
and produces an irreducible lattice for every $m\ge1$.
We establish these properties using codimension-one subcodes of
extremal ternary self-dual codes of length $24$.

As an application, we answer a question of Cao, Koolen, Liu, and
Yang~\cite[Problem~1.9]{CKLY} affirmatively.

\begin{corollary}\label{cor:main}
	The lattices $L_m$ yield pairwise switching-inequivalent connected
	non-extendable signed graphs $\Gamma_m$ with smallest eigenvalue $-3$.
\end{corollary}

\section{Preliminaries}

A code $C$ over $\mathbb F_3$ is called \emph{ternary}.
All codes in this paper are ternary linear codes.
The \emph{weight} $\wt(\cc)$ of a word $\cc\in\F^n$ is the number of its nonzero coordinates.
A linear code of length $n$, dimension $k$, and minimum nonzero
weight $d$ is called an $[n,k,d]$ code.
For a linear code $C$ of length $n$, we write
\[
	A_i(C):=\bigl|\{\cc\in C:\wt(\cc)=i\}\bigr|
	\qquad (0\leq i\leq n).
\]
A code $C$ of length $n$ is said to be \emph{self-dual} if
$C=C^\perp$, where the dual code $C^\perp$ of $C$ is defined as
$C^\perp
	=
	\{\x\in\mathbb F_3^n \mid \x\cdot \y=0
	\text{ for all }\y\in C\}$.
It was shown in~\cite{MallowsSloane1973} that the minimum weight $d$ of a self-dual
code of length $n$ is bounded by $d\le 3\left\lfloor\frac{n}{12}\right\rfloor+3$.
If $d= 3\left\lfloor\frac{n}{12}\right\rfloor+3$, then the code is called \emph{extremal}.
Two codes $C$ and $C'$ are \emph{equivalent} if there exists a
monomial matrix $P$ with
$C'=C\cdot P=\{xP\mid x\in C\}$, and \emph{inequivalent} otherwise.

Leon, Pless, and Sloane~\cite[Theorem~6]{LeonPlessSloane} showed that there are precisely two equivalence classes of
extremal ternary self-dual codes of length $24$, represented by the extended
ternary quadratic-residue code $Q_{24}$ and the Pless symmetry code $P_{24}$.

The distinct supports of the weight-$9$ codewords of $Q_{24}$ and
$P_{24}$ form two nonisomorphic $5$-$(24,9,6)$ designs
\cite{AssmusMattson,Pless1972}.
Consequently, by the above classification, the distinct supports of
the weight-$9$ codewords of every extremal ternary self-dual
$[24,12,9]$ code form a $5$-$(24,9,6)$ design.
For $0\leq j\leq5$, let $\lambda_j$ denote the number of blocks
containing a fixed $j$-subset of the coordinate set.
The derived parameters are
\begin{equation}\label{eq:lambdas}
	(\lambda_0,\lambda_1,\lambda_2,\lambda_3,\lambda_4,\lambda_5)
	=(2024,759,264,84,24,6).
\end{equation}

Throughout, inner products of codewords are taken over $\F$, and
we use $2=-1$. If $\cc\in\F^n$, its \emph{canonical integral lift}
$\widetilde \cc\in\{0,\pm1\}^n$ is obtained by replacing the symbol $2$ by
$-1$.

For later use, let
\begin{equation}\label{eq:krawtchouk}
	K_j(i):=\sum_{\ell=0}^j(-1)^\ell2^{j-\ell}
	\binom{i}{\ell}\binom{24-i}{j-\ell}
\end{equation}
be the ternary Krawtchouk polynomials of length $24$.
If $E$ is a ternary
$[24,r]$ code, $A_i=A_i(E)$, and $B_j=A_j(E^\perp)$, then
\begin{equation}\label{eq:macwilliams}
	\sum_{i=0}^{24}A_iK_j(i)=3^rB_j.
\end{equation}
This is the Krawtchouk form of the MacWilliams identities; see
\cite[Chapter~7]{HuffmanPless}.

An $n$-dimensional lattice $L$ is \emph{integral} if
$L\subseteq L^*$, where the dual lattice $L^*$ is defined as
$L^*
	=
	\{\x\in\mathbb R^n \mid \x\cdot \y\in\mathbb Z
	\text{ for all }\y\in L\}$.
A lattice $L$ with $L=L^*$ is called \emph{unimodular}.
A unimodular lattice $L$ is \emph{even} if all vectors of $L$
have even norms, and \emph{odd} if some vector has an odd norm.
A lattice $L$ is called \emph{$s$-integrable} if $\sqrt{s}\,L$ is isometric
to a sublattice of $\mathbb Z^N$ for some positive integer $N$.
A lattice is \emph{irreducible} if it admits no orthogonal decomposition into two nonzero sublattices.
A \emph{root} is a vector of norm $2$, and a lattice is \emph{rootless} if it contains no roots.
We write $\Min(L)$ for the set of vectors attaining the minimum
nonzero norm of $L$.
Generation by minimal vectors means
$L=\langle\Min(L)\rangle_{\mathbb Z}$.

For a ternary linear code $C$ of length $n$, its
\emph{Construction~A lattice} is
\begin{equation}\label{eq:defLm}
	\mathcal{A}_3(C)
	:=
	\frac1{\sqrt3}
	\{\zz\in\mathbb Z^n:\zz\bmod3\in C\}.
\end{equation}
If $C$ is self-dual and has minimum weight $d$, then
$\mathcal{A}_3(C)$ is unimodular and $3$-integrable with minimum norm
$\min\{3,d/3\}$; see \cite[Section~2]{HaradaMunemasa}.
In particular, $\mathcal{A}_3(P_{24})$ and
$\mathcal{A}_3(Q_{24})$ are both isometric to the odd Leech lattice
$\mathcal O_{24}$, the unique odd unimodular
lattice of rank $24$ and minimum norm $3$;
see \cite[Chapter~17]{ConwaySloane}.

\begin{lemma}[{\cite[Lemma~5]{HaradaMunemasa}}]\label{lem:HM}
	Let $C$ be a ternary self-dual code of minimum weight at
	least $6$.
	Then $C$ is decomposable if and only if $\mathcal{A}_3(C)$ is
	decomposable.
\end{lemma}

\section{A ternary code seed}\label{sec:seed}

We start from an extremal ternary self-dual $[24,12,9]$ code $C$.
Choose a word $\w\in C$ of weight $9$. After a monomial transformation, we
may assume
$ \w=\mathbf1_S$ and $|S|=9.$
Fix any partition $S=P\sqcup Q$ with $|P|=4$ and $|Q|=5$, and put
\begin{equation*}
	\p=-\mathbf1_P,
	\qquad
	\q=-\mathbf1_Q.
\end{equation*}
Then
\begin{equation*}
	\w+\p+\q=0,
	\quad \p\cdot \p=1,
	\quad \q\cdot \q=2,
	\quad \p\cdot \q=0,
	\quad \w\cdot \p=2.
\end{equation*}
Define
\begin{equation*}
	D=C\cap \p^\perp,
	\qquad
	D_9=\{\dd\in D:\wt(\dd)=9\}.
\end{equation*}
For a family of codewords, its \emph{support graph} has the coordinate
positions as vertices, with two coordinates adjacent when they occur together
in the support of one member of the family.

\begin{theorem}\label{thm:seed}
	For every choice of $C,\w,P,Q$ above,
	the following statements hold:
	\begin{enumerate}[label=\textnormal{(\roman*)}]
		\item $\dim D=11$, $C=D\oplus\langle \w\rangle$, and $D\perp \q$;
		\item $D=\langle D_9\rangle$;
		\item the support graph of $D_9$ is connected.
	\end{enumerate}
\end{theorem}

\begin{proof}
	The functional $\cc\mapsto \cc\cdot \p$ is nonzero on $C$, since
	$\w\cdot \p=2$.
	Hence its kernel $D$ has dimension $11$, and
	$C=D\oplus\langle \w\rangle$. Moreover, $D\perp \p$, and since
	$\q=-\w-\p$ and $D\subseteq C=C^\perp$, also $D\perp \q$.
	This proves (i).

	Put $E=\langle D_9\rangle$. We first show that $E$ has full support. If
	$j\notin P$, inclusion--exclusion in the $5$-$(24,9,6)$ design shows that
	the number of blocks containing $j$ and avoiding $P$ is
	\[
		\lambda_1-4\lambda_2+6\lambda_3-4\lambda_4+\lambda_5
		=759-4\cdot264+6\cdot84-4\cdot24+6=117,
	\]
	by \eqref{eq:lambdas}.
	Every corresponding weight-$9$ word lies in $D_9$.
	If $j\in P$, let $\nu_j$ be the number of blocks containing the $6$-set
	$\{j\}\cup Q$. The number of blocks containing $j$ and avoiding $Q$ is
	\[
		\lambda_1-5\lambda_2+10\lambda_3-10\lambda_4+5\lambda_5-\nu_j
		=69-\nu_j\ge63,
	\]
	since $\nu_j\le\lambda_5=6$. For a word $\cc$ on such a block,
	$\cc\cdot \w=0$ and $\cc$ vanishes on $Q$, whence $\cc\cdot \p=0$.
	Thus every
	coordinate occurs in a word of $D_9$.

	We next establish a lower bound on $|D_9|$. Let
	$a=A_9(D)$ and $ B_j=A_j(D^\perp).$
	Since
	$
		D^\perp=C^\perp+\langle \p\rangle=C+\langle \p\rangle,
	$
	the only nonzero words of $D^\perp$ of weight at most $4$ are $\p$ and $-\p$.
	Indeed, if $\cc\in C$ and $\alpha\ne0$ satisfy
	$\wt(\cc+\alpha \p)\le4$, then $\wt(\cc)\le8$, so $\cc=0$.
	Hence
	\[
		B_0=1,
		\qquad B_1=B_2=B_3=0,
		\qquad B_4=2.
	\]
	Since $D\subset C$, all weights in $D$ are divisible by $3$, and its nonzero weights are at least
	$9$. Substitution in the standard identities \eqref{eq:macwilliams} for
	$j=0,1,2,3,4$ gives
	\[
		A_{12}(D)=27362-5a,~
		A_{15}(D)=67400+10a,~
		A_{18}(D)=79748-10a,\]
	\[
		A_{21}(D)=1240+5a,~
		A_{24}(D)=1396-a.
	\]
	The equation for $j=5$ then gives
	$ B_5=\frac{a-1360}{3}.$
	Consequently
	\begin{equation}\label{eq:D9lower}
		|D_9|=a\ge1360.
	\end{equation}

	It remains to show that $\dim E=11$.
	We use the following Delsarte-type linear-programming argument \cite{Delsarte}.
	Set
	\begin{equation}\label{eq:fi}
		F(i)=\frac{1532+404K_1(i)-53K_5(i)-11K_6(i)}{228906}.
	\end{equation}
	Direct evaluation from \eqref{eq:krawtchouk} gives
	\begin{equation}\label{eq:Fvalues}
		\begin{array}{c|rrrrrrr}
			i    & 0                    & 9 & 12 & 15 & 18 & 21           & 24                \\[0.3ex]
			\hline                                                                            \\[-1.5ex]
			F(i) & -\dfrac{114418}{157} & 1 & 0  & 0  & 0  & \dfrac1{314} & \dfrac{524}{157}.
		\end{array}
	\end{equation}
	Let $r=\dim E$ and write $B_j'=A_j(E^\perp)$. Since $E$ has full support, $B_1'=0$.
	Since the nonzero weights of $E$ belong to
	$\{9,12,15,18,21,24\}$, \eqref{eq:Fvalues} gives
	\[
		\sum_{i>0}A_i(E)F(i)
		=
		A_9(E)+\frac{1}{314}A_{21}(E)
		+\frac{524}{157}A_{24}(E)
		\ge A_9(E).
	\]
	On the other hand, by \eqref{eq:fi}, the MacWilliams identities \eqref{eq:macwilliams} yield
	\begin{align*}
		\sum_{i=0}^{24}A_i(E)F(i)
		=
		\frac{3^r}{228906}
		\bigl(1532B_0'+404B_1'-53B_5'-11B_6'\bigr)\le
		\frac{1532\cdot3^r}{228906},
	\end{align*}
	because $B_0'=1$, $B_1'=0$, and $B_5',B_6'\ge0$.
	Since $A_0(E)=1$ and
	$F(0)=-114418/157$, it follows that
	\[
		A_9(E)
		\le
		\sum_{i>0}A_i(E)F(i)
		\le
		\frac{1532\cdot3^r}{228906}
		+\frac{114418}{157}.
	\]
	Note that $A_9(E)=|D_9|\ge1360$ by \eqref{eq:D9lower}, and when $r\leq 10$, $A_9(E)<1360$.
	Thus
	$r\ge11$.
	Since $E\subseteq D$ and $\dim D=11$, we have $E=D$. This proves (ii).

	For (iii), if the support graph had $t\ge2$ components, generation would give
	$D=\bigoplus_{i=1}^t D_i$, with nonzero $[n_i,k_i,d_i]_3$ codes satisfying
	$d_i\ge9$. The Singleton bound gives
	\[
		13=24-11=\sum_i(n_i-k_i)\ge\sum_i(d_i-1)\ge8t\ge16,
	\]
	a contradiction.
	This completes the proof.
\end{proof}

\begin{lemma}\label{lem:local}
	For $a,b\in\F$, put
	\(
	\mu(a,b)=\min\{\wt(\dd+a\q+b\p):\dd\in D\}.
	\)
	With rows and columns indexed by $0,1,2$,
	the following inequality holds entrywise:
	\begin{equation}\label{eq:local}
		(\mu(a,b))_{a,b\in\F}\ge
		\begin{pmatrix}
			0 & 4 & 4 \\
			5 & 9 & 6 \\
			5 & 6 & 9
		\end{pmatrix}
	\end{equation}
\end{lemma}

\begin{proof}
	Let $\uu=\dd+a\q+b\p$ with $\dd\in D$.
	Since $D$ is orthogonal to both $\p$ and $\q$, and $D$ is
	self-orthogonal,
	\begin{equation}\label{eq:weightcong}
		\wt(\uu)\equiv \uu\cdot \uu\equiv2a^2+b^2\pmod3.
	\end{equation}
	The first row and column follow from $\wt(\p)=4$, $\wt(\q)=5$,
	$d(C)=9$, and \eqref{eq:weightcong}. If $a=b\ne0$, then
	$a\q+b\p=-a\w$, so $\dd-a\w$ is a nonzero word of $C$ and has weight at least $9$.
	If $a=-b\ne0$, scalar multiplication reduces to $(a,b)=(1,2)$. For
	$\uu=\dd+\q-\p$, the word $\uu-\p=\dd-\w$ lies nontrivially in $C$, so
	$\wt(\uu)\ge9-4=5$; by \eqref{eq:weightcong} its weight is divisible by $3$,
	and hence is at least $6$.
\end{proof}

\section{Construction of the lattices $L_m$}\label{sec:lattice-graph}

For every integer $m\ge 1$, we first construct an indecomposable
ternary self-dual $[24m,12m,9]$ code $C_m$ generated by its minimum-weight
codewords.

Fix $C,\w,P,Q,\p,\q,D$ as in \cref{sec:seed}.
For $m\ge1$, index $m$ blocks of
length $24$ by $i\in\Zcyc m$.
For $\vv\in\F^{24}$, write $\vv^{(i)}$ for the
vector supported on block $i$ with value $\vv$. Define
\begin{equation}\label{eq:defCm}
	\x_i=\p^{(i)}+\q^{(i+1)},
	\qquad
	X_m=\langle \x_i:i\in\Zcyc m\rangle,
	\qquad
	C_m=D^{\oplus m}+X_m.
\end{equation}
For $m=1$, $\x_0=\p+\q=-\w$, so $C_1=C$.

\begin{theorem}\label{thm:Cm}
	For every $m\ge1$, $C_m$ is an indecomposable ternary self-dual code with
	parameters $[24m,12m,9]$.
	It is generated by weight-$9$ words whose support
	graph is connected.
\end{theorem}

\begin{proof}
	The code $D^{\oplus m}$ is self-orthogonal and is orthogonal to every bridge.
	Moreover,
	$ \x_i\cdot \x_i=0$,
	and distinct bridges are orthogonal.
	Hence $C_m$ is self-orthogonal.

	If $\sum_i\alpha_i \x_i\in D^{\oplus m}$, then its block-$i$ component is
	$\alpha_i \p+\alpha_{i-1}\q$.
	Taking its inner product with $\p$ gives
	$\alpha_i=0$.
	Thus the $m$ bridges are independent and
	$X_m\cap D^{\oplus m}=0$.
	Therefore
	\(
	\dim C_m=11m+m=12m,
	\)
	so $C_m$ is self-dual.

	Write a word of $C_m$ as
	\(
	\zz=(\dd_i)_{i\in\Zcyc m}+\sum_i\alpha_i \x_i.
	\)
	Its block-$i$ component is
	\(
	\zz_i=\dd_i+\alpha_{i-1}\q+\alpha_i \p.
	\)
	The weight of $\zz$ satisfies
	$\wt(\mathbf z)=\sum_i\wt(\mathbf z_i)\geq\sum_i\mu(\alpha_{i-1},\alpha_i).$
	Therefore, the matrix \eqref{eq:local} assigns a lower-bound cost to each transition
	$\alpha_{i-1}\to\alpha_i$ in the cyclic state sequence. If all states are
	zero, every nonzero word has weight at least $9$. If both zero and nonzero
	states occur, every maximal nonzero run has entrance cost at least $4$ and
	exit cost at least $5$. If all states are nonzero, either they are all equal,
	in which case a transition already costs at least $9$, or both $1$ and $2$
	occur, in which case the cycle has at least two transitions of cost at least
	$6$. Hence $d(C_m)\ge9$. Each bridge has weight $9$, so $d(C_m)=9$.

	By Theorem~\ref{thm:seed}, the block-local copies of $D_9$, together with the
	bridges, generate $C_m$; all have weight $9$. Within each block their support
	graph is connected, and every bridge joins two consecutive blocks. Thus the
	full support graph is connected.

	Finally, suppose that $C_m$ decomposes over a nontrivial coordinate partition
	as $U\oplus V$. Since $d(C_m)=9$, a weight-$9$ word cannot have nonzero
	projections in both summands.
	This contradicts the connectedness of the support graph.
	Thus $C_m$ is indecomposable.
\end{proof}

We now prove Theorem~\ref{thm:main lattice} by taking $L_m:=\mathcal{A}_3(C_m)$ to be
the Construction~A lattice associated with $C_m$.
By~\eqref{eq:defLm}, $L_m$ is $3$-integrable.

\begin{proof}[Proof of Theorem~\ref{thm:main lattice}]
	The Construction~A properties recalled in Section~2 show that $L_m$
	is unimodular of rank $24m$ and minimum norm $3$.
	Lemma~\ref{lem:HM} gives irreducibility, and a vector of norm $3$
	shows that $L_m$ is odd.

	The vectors $\sqrt3\e_j$ are minimal and generate $\sqrt3\mathbb Z^{24m}$.
	For every weight-$9$ word $\cc\in C_m$, its canonical lift
	$\widetilde \cc/\sqrt3$ is also minimal. Since the weight-$9$ words generate
	$C_m$ and
	$ L_m/\sqrt3\mathbb Z^{24m}\cong C_m,$
	these minimal vectors generate $L_m$.
	This completes the proof.
\end{proof}

\begin{remark}\label{rem:four-lattice-types}
A computer-assisted classification shows that, for each fixed
$m\ge2$, our construction yields exactly four isometry classes
of lattices, listed in Table~\ref{tab:four-lattice-types}.
	Put
	$a:=A_9(D)$ and $
	b:=\bigl|\{\cc\in D+\q:\wt(\cc)=5\}\bigr|.$
	Their kissing numbers are given by
	\[
	|\Min(L_m)|=48m+A_9(C_m)=m(48+a+2b),
	\]
	and distinguish the four classes.
	For $m=1$, all choices give the odd Leech lattice.
\end{remark}

\begin{table}[H]
	\centering
	\caption{The four lattice isometry classes arising from the
		extremal $4+5$ construction for each fixed $m\ge2$.}
	\label{tab:four-lattice-types}
	\begin{tabular}{@{}crrcl@{}}
		\toprule
		Type & $a=A_9(D)$ & $b$ & $|\Min(L_m)|$ & Starting code $C$ \\
		\midrule
		I   & $1372$ & $2$ & $1424m$ & $Q_{24}$ or $P_{24}$ \\
		II  & $1378$ & $3$ & $1432m$ & $Q_{24}$ or $P_{24}$ \\
		III & $1384$ & $4$ & $1440m$ & $Q_{24}$ or $P_{24}$ \\
		IV  & $1396$ & $6$ & $1456m$ & $P_{24}$ \\
		\bottomrule
	\end{tabular}
\end{table}

All lattices constructed here are $3$-integrable.
This suggests the following question.

\begin{question}\label{ques:bounded-integrability}
	Does there exist a positive integer $t$ such that every integral
	lattice generated by vectors of norm $3$ is $t$-integrable?
\end{question}

\section{Non-extendable signed graphs}\label{sec:non-extendable}

A \emph{signed graph} is a finite simple graph $G$ together with a
map $\sigma:E(G)\to\{\pm1\}$. Its signed adjacency matrix
$A=A(G,\sigma)$ has entry $\sigma(\{x,y\})$ when $x$ and $y$ are
adjacent, and $0$ otherwise. Its eigenvalues are those of $A$.
Switching replaces $A$ by $TAT$, where $T$ is diagonal with entries
in $\{\pm1\}$. We consider switching classes up to vertex relabelling.
Following \cite[Definition~1.4]{CKLY}, a connected signed graph with
smallest eigenvalue $-3$ is \emph{non-extendable} if it is not a proper
induced signed subgraph of any connected signed graph with smallest
eigenvalue at least $-3$.

We prove Corollary~\ref{cor:main} by constructing $\Gamma_m$ from one
representative of each antipodal pair of minimal vectors of $L_m$.
To prove non-extendability, we use a projection argument similar to
that in the proof of \cite[Lemma~2.2]{CKLY}.

\begin{proof}[Proof of Corollary~\ref{cor:main}]
	Choose one representative from each antipodal pair
	$\{\vv,-\vv\}\subseteq\Min(L_m)$, and denote the resulting set by $S_m$.
	For distinct $\uu,\vv\in S_m$, the vectors $\uu\pm\vv$ are nonzero, so
	\(
	3\le (\uu\pm\vv)\cdot(\uu\pm\vv)
	=6\pm2\uu\cdot\vv.
	\)
	Integrality therefore gives $\uu\cdot\vv\in\{0,\pm1\}$.
	Define a signed graph $\Gamma_m$ on $S_m$ by joining distinct vectors
	when their inner product is nonzero, with this inner product as the edge sign.
	Writing $A_m=A(\Gamma_m)$, we have
	\(
	3I+A_m=\operatorname{Gram}(S_m)\succeq0\).
	Since $S_m$ generates $L_m$, \( \rank(3I+A_m)=24m.\)
	There are $24m$ antipodal pairs of coordinate minimal vectors
	$\pm\sqrt3\e_j$, and at least one further pair arising from a
	weight-$9$ word of $C_m$. Hence $|S_m|>24m$, so $3I+A_m$ is singular
	and $\lambda_{\min}(\Gamma_m)=-3$.
	Different choices of representatives only switch $\Gamma_m$.
	Moreover, if $\Gamma_m$ were disconnected, the integer spans of its
	components would give a nontrivial orthogonal decomposition of $L_m$.
	Thus $\Gamma_m$ is connected.

	Suppose that a connected signed graph $\Gamma'$ properly contains
	$\Gamma_m$ as an induced signed subgraph and has smallest eigenvalue
	at least $-3$. Choose a Gram representation of $3I+A(\Gamma')$.
	Identify the span of the old Gram vectors isometrically with
	$V=\operatorname{span}_{\mathbb R}L_m$, so that the old vectors are
	precisely those of $S_m$.
	Since $\Gamma'$ is connected, some new vertex is adjacent to an old
	vertex. Let $\y$ be its Gram vector and let $\zz$ be its orthogonal
	projection onto $V$. Then $\zz\ne0$ and
	$\zz\cdot\vv=\y\cdot\vv\in\{0,\pm1\}$ for every $\vv\in S_m$.
	As $S_m$ generates $L_m$, we have $\zz\in L_m^*=L_m$.
	The minimum norm and the projection inequality give
	\[
		3\le\zz\cdot\zz\le\y\cdot\y=3.
	\]
	Equality forces $\y=\zz\in\Min(L_m)$, and hence $\y=\pm\vv$ for
	some $\vv\in S_m$. This gives the impossible off-diagonal Gram entry
	$\y\cdot\vv=\pm3$.
	Therefore $\Gamma_m$ is non-extendable.

	Finally, $\rank(A_m+3I)=24m$ is invariant under switching and vertex
	relabelling. Different values of $m$ therefore give pairwise distinct
	switching classes.
\end{proof}

\noindent{\bf Statement of AI use.}
ChatGPT 5.6 Sol initially provided a specific ternary self-dual $[24,12,9]$ code,
together with a computer-assisted argument showing that this code
gives rise to an infinite family of lattices.
Building on this, we generalized the construction to arbitrary
ternary self-dual $[24,12,9]$ codes and gave a proof that does not
rely on computer-assisted verification.
The authors take full responsibility for the mathematical
claims, proofs, and final presentation of the paper.

\end{document}